\documentclass{amsart}
\usepackage{hyperref,enumerate}
\usepackage{amsmath,amsfonts,amsthm,amscd,mathabx}
\usepackage{esint}
\newcommand{\bg}{\boldsymbol{g}}

\newcommand{\tfP}{\mathring{P}}

\DeclareMathOperator{\tr}{tr}
\DeclareMathOperator{\Ima}{Im}

\newtheorem{theorem}{Theorem}[section]
\newtheorem{lemma}[theorem]{Lemma}
\newtheorem{claim}[theorem]{Claim}
\newtheorem{proposition}[theorem]{Proposition}
\newtheorem{corollary}[theorem]{Corollary}

\newtheorem{conjecture}[theorem]{Conjecture}
\theoremstyle{definition}
\newtheorem{definition}[theorem]{Definition}

\newtheorem{remark}[theorem]{Remark}
\numberwithin{equation}{section}


\title[The Anti-Self-Dual Deformation Complex]{The Anti-Self-Dual Deformation Complex and a conjecture of Singer}
\author{A. Rod Gover}
\address{Department of Mathematics \\
         University of Auckland \\
         Private Bag 92019, Auckland 1142, New Zealand}
\author{Matthew J. Gursky}
\address{Department of Mathematics \\
         University of Notre Dame\\
         Notre Dame, IN 46556}
\begin{document}

\maketitle

\begin{abstract}  Let $(M^4,g)$ be a smooth, closed, oriented anti-self-dual (ASD) four-manifold.  $(M^4,g)$ is said to be {\em unobstructed} if the cokernel of the linearization of the self-dual Weyl tensor is trivial.  This condition can also be characterized as the vanishing of the second cohomology group of the ASD deformation complex, and is central to understanding the local structure of the moduli space of ASD conformal structures.  It also arises in construction of ASD manifolds by twistor and gluing methods.  In this article we give conformally invariant conditions which imply an ASD manifold of positive Yamabe type is unobstructed.

\end{abstract}

%%%%%%%%%%%%
\section{Introduction}
%%%%%%%%%%%%%%

Let $M^4$ be a smooth, closed, oriented four-manifold.  Given a Riemannian metric $g$ on $M^4$, the bundle of two-forms $\Lambda^2 = \Lambda^2(M^4)$ splits into the sub-bundles of self-dual and anti-self=dual two-forms under the action of the Hodge $\star$=operator:
\begin{align*}
\Lambda^2 = \Lambda^2_{+} \oplus \Lambda^2_{-}.
\end{align*}
By a result of Singer-Thorpe \cite{ST}, the curvature operator $Rm : \Lambda^2 \rightarrow \Lambda^2$ has a canonical block decomposition of the form
	\begin{align*}
	Rm = \left(\begin{array}{ll}
	A^{+} &  B  \\
	B^t & A^{-}
	\end{array}
\right),
	\end{align*}
where $A^{\pm} : \Lambda^2_{\pm} \rightarrow \Lambda^2_{\pm}$ and $B : \Lambda^2_{+} \rightarrow \Lambda^2_{-}$.  If $W$ denotes the Weyl tensor, then $W^{\pm} : \Lambda^2_{\pm} \rightarrow \Lambda^2_{\pm}$, and
\begin{align*}
A^{\pm} = W^{\pm} + \frac{1}{12} R \, I,
\end{align*}
where $I$ is the identity and $R$ is the scalar curvature.

\begin{definition} We say that $(M^4,g)$ is {\em anti-self-dual} (ASD) if $W^{+}_g \equiv 0$.   \end{definition}

The notion of (anti-)self-duality is conformally invariant: if $W^{+}_g = 0$ for a metric $g$ and $\tilde{g} = e^f g$, then $W^{+}_{\tilde{g}} = 0$.  This property will be crucial for the proof of our main result below.

There are topological obstructions to the existence of ASD metrics.  By the Hirzebruch signature formula,
\begin{align} \label{SF}
48 \pi^2 \tau(M^4) = \int \left( |W^{+}_{g}|^2 - |W^{-}_{g}|^2 \right) \, dv_{g},
\end{align}
where $\tau(M^4)$ is the signature of the intersection form on $H^2_{dR}(M^4)$.  In particular, we see that if $(M^4,g)$ is ASD then $\tau(M^4) \leq 0$, with equality if and only if $g$ is LCF.  If $(M^4,g)$ is ASD with positive scalar curvature, then the intersection form is actually definite (see Proposition 1 of \cite{LeBrunPAMS}).  To see this, we first observe that the splitting of $\Lambda^2(M^4)$ induces a splitting on the space of harmonic two-forms, hence $H^2_{dR}(M^4) = H^2_{+}(M^4) \oplus H^2_{-}(M^4)$.  If $\omega \in H^2_{+}(M^4)$, then the Weitzenb\"ock formula for the Hodge laplacian $\Delta_2$ is given by
\begin{align} \label{WFHodge}
\Delta_2 \omega = \Delta \omega + 2 W^{+}(\omega) - \frac{1}{3}R \omega,
\end{align}
where $\Delta = g^{ij} \nabla_i \nabla_j$ is the rough laplacian.  If $(M^4,g)$ is ASD and the scalar curvature $R > 0$, then (\ref{WFHodge}) immediately implies $\omega = 0$.

Examples of ASD manifolds include locally conformally flat (LCF) manifolds, since in dimensions greater than three LCF is equivalent to the vanishing of the Weyl tensor.  In particular, $S^4$ endowed with the round metric $g_c$ is ASD.  Non-simply connected examples include the product metric on $S^3 \times S^1$, and more generally the metrics constructed via gluing on the connected sums $S^3 \times S^1 \hash \cdots \hash S^3 \times S^1$.

A non-LCF example is given by complex projective space with the Fubini-Study metric, and we take the opposite of the its natural orientation as a complex manifold; i.e., $(-\mathbb{CP}^2, g_{FS})$.   Any scalar-flat K\"ahler metric is also ASD, since the K\"ahler condition implies that $W^{+}$ is determined by $R$ (see \cite{Derdzinski}, Section 3).   There are many constructions ASD manifolds in the literature; see for example \cite{Poon86}, \cite{Poon92}, \cite{Leb91}, \cite{DF}, \cite{Floer}, \cite{Taubes}, and the references in Lecture 6 of \cite{Via}.

Roughly speaking, the constructions of ASD manifolds are based on either `twistor' or `analytic' methods.  The former approach relies on the so-called Penrose correspondence, which will play no role in our work but is of profound importance in the study of ASD manifolds.  Briefly, the unit sphere bundle $\mathcal{Z}$ of $\Lambda^2_{+}(M^4)$ carries a canonical complex structure.  As shown in \cite{AHS78}, this complex structure is integrable if and only if the metric is ASD.   Therefore, we can associate to any ASD manifold a complex manifold of (complex) dimension three, called the {\em twistor space} of $(M^4,g)$.  This important observation allows one to use methods of complex geometry to study the existence and deformation theory of ASD conformal structures.

Analytic methods involve the construction of an ASD metric on the connected sum of two manifolds admitting ASD metrics via perturbative methods.  As in other geometric gluing constructions, if $(M_1,g_1)$ and $(M_2, g_2)$ are ASD manifolds, one first constructs a metric $h$ on the connected sum $M_1 \hash M_2$ which is ``approximately'' ASD; i.e., $W^{+}(h)$ is small in some appropriately defined norm.  This reduces the problem to the study of the mapping properties of the linearized operator, in order to perturb $h$ to produce an actual ASD metric.  To make this more precise, we now introduce the ASD deformation complex.

\smallskip

\subsection{The ASD deformation complex}

Let $\mathcal{M}(M^4)$ be the space of smooth Riemannian metrics on $M^4$, and $\mathcal{R}(M^4)$ the bundle of algebraic curvature tensors.  We can view $W^{+}$ as a mapping
\begin{align*}
W^{+} : \mathcal{M}(M^4) \rightarrow \mathcal{R}(M^4).
\end{align*}
Let $g \in \mathcal{M}(M^4)$ be an ASD metric.  We can identify the formal tangent space of $\mathcal{M}$ at $g$ with sections of the bundle of symmetric two-tensors, $S^2(T^{*}M^4)$.  Let
\begin{align}
\mathcal{D} : \Gamma( S^2(T^{*}M^4) ) \rightarrow \Gamma( \mathcal{R}(M^4))
\end{align}
denote the linearization of $W^{+}$ at $g$; i.e., for $h \in \Gamma(S^2(T^{*}M^4))$,
\begin{align*}
\mathcal{D}_g h = \frac{d}{ds} W^{+}(g + s h) \big|_{s=0}.
\end{align*}
The choice of a conformal class of metrics $[g]$ determines the bundle of algebraic Weyl tensors, $\mathcal{W} = \mathcal{W}(M^4,[g])$, and the sub-bundles $\mathcal{W}^{\pm} \subset \mathcal{W} \subset S^2_0(\Lambda^2_{+})$, where $S^2_0(\Lambda^2_{+})$ is the bundle of symmetric, trace-free endomorphisms of $\Lambda^2_{+}$.  Note that
\begin{align} \label{Ddef0}
\mathcal{D}_g : \Gamma( S^2(T^{*}M^4)) \rightarrow \Gamma( \mathcal{W}^{+}) \subset \Gamma( S^2_0(\Lambda^2_{+})).
\end{align}
In fact, since $g$ is ASD, by conformal invariance $\mathcal{D}_g (f g) = 0$ for any $f \in C^{\infty}(M^4)$, hence
\begin{align} \label{Ddef}
\mathcal{D}_g : \Gamma( S^2_0(T^{*}M^4) ) \rightarrow \Gamma(\mathcal{W}^{+}).
\end{align}
We also let
\begin{align} \label{dstar}
\mathcal{D}^{*}_g : \Gamma(\mathcal{W}^{+}) \rightarrow \Gamma( S^2_0(T^{*}M^4) )
\end{align}
denote the $L^2$-formal adjoint of $\mathcal{D}_g$.  Although the formula for $\mathcal{D}_g$ is somewhat involved, the formula for $\mathcal{D}^{*}$ is much more compact (see Proposition A.4 of \cite{Itoh93}):
\begin{align} \label{Dstar}
\mathcal{D}^{*}U_{ij} = 2 \left( \nabla^k \nabla^{\ell} U_{ikj\ell} + P^{k \ell}U_{ikj\ell}\right),
\end{align}
where $P$ denotes the Schouten tensor (see Section \ref{TractorSec}).

Let $\mathcal{K}_g : \Gamma(T^{*}M^4) \rightarrow S^2_0(T^{*}M^4)$ denote the Killing operator:
\begin{align*}
\mathcal{K}_g(\omega)_{ij} = \nabla_i \omega_j + \nabla_j \omega_i - \frac{1}{2} (\delta_g \omega) \, g_{ij},
\end{align*}
where $\delta_g \omega = \nabla^k \omega_k$ is the divergence of $\omega$.  The kernel of $\mathcal{K}$ consists of those one-forms whose
dual vector field are conformal Killing.  Moreover, by diffeomorphism and conformal invariance,
\begin{align*}
\Ima \mathcal{K} \subset \ker \mathcal{D}.
\end{align*}

The {\em ASD deformation complex} is given by
\begin{align} \label{ASDcx}
\Gamma(T^{*}M^4) \xrightarrow{\mathcal{K}} \Gamma( S^2_0(T^{*}M^4)) \xrightarrow{\mathcal{D}} \Gamma( S^2_0(\Lambda^2_{+})).
\end{align}
This complex is elliptic; see \cite{KK}, Section 2.  The associated cohomology groups are given by
\begin{align} \label{coho} \begin{split}
H^0_{ASD}(M^4,g) &= \ker \mathcal{K}_g, \\
H^1_{ASD}(M^4,g) &= \{ h \in \Gamma( S^2_0(T^{*}M^4)) \, : \, \delta_g h = 0, \ \mathcal{D}_g h = 0 \}, \\
H^2_{ASD}(M^4,g) &= \ker \mathcal{D}^{*}_g,
\end{split}
\end{align}
where for $h \in \Gamma( S^2(T^{*}M^4))$, $(\delta_g h)_j = g^{ik}\nabla_k h_{ij}$ is the divergence.  The Atiyah-Singer index theorem can
be used to calculate the index of (\ref{ASDcx})\footnote{This was calculation was done by Singer in unpublished notes, but can be found in the literature for example in \cite{KK}}:
\begin{align*}
\mbox{Ind}_{ASD}(M^4,g) = \frac{1}{2} \left( 15 \chi\left(M^4\right) + 29 \tau\left(M^4\right) \right).
\end{align*}

Vanishing of the cohomology groups also provides information on the local structure of the moduli space of
ASD conformal structures:

\begin{proposition} (See \cite{Itoh93}, \cite{Via}) Suppose $(M^4,g)$ is ASD with
\begin{align*}
H^0_{ASD}(M^4,g) = \{ 0 \}, \ \ H^2_{ASD}(M^4,g) = \{ 0 \}.
\end{align*}
Then the moduli space of anti-self-dual conformal structures near $g$ is a smooth, finite-dimensional manifold of dimension $\dim H^1_{ASD}(M^4,g)$.
\end{proposition}

This leads to the following definition:

\begin{definition}  Let $(M^4,g)$ be ASD.  We say that $(M^4,g)$ is {\em unobstructed} if $H^2_{ASD}(M^4,g) = \{ 0 \}$.
\end{definition}

By the work of Floer \cite{Floer} and Donaldson-Friedman \cite{DF}, if $(M_1, g_1)$ and $(M_2, g_2)$ are unobstructed ASD manifolds, then the connected sum $M_1 \hash M_2$ admits an ASD metric.  Thus, we are lead to the question:  under what condition is an ASD manifold unobstructed?  The following conjecture is often attributed to Singer:

\begin{conjecture} Let $(M^4,g)$ be ASD. If the Yamabe invariant of $(M^4,g)$ is positive, then $(M^4,g)$ is unobstructed.
\end{conjecture}

For ASD Einstein manifolds the operators $\mathcal{D}\mathcal{D}^{*}$ and $\mathcal{D}^{*}\mathcal{D}$ were explicitly computed in \cite{Itoh} and \cite{Kob}.  It follows from these calculations (and can also be seen by more or less direct calculation) that $(S^4,g_0)$ and $(\mathbb{CP}^2,g_{FS})$ are unobstructed.  We remark that for non-Einstein ASD manifolds, the formulas for these operators are fairly intractable.

The vanishing of $H^2_{ASD}(M^4,[g])$ can sometimes be verified when the twistor space is explicitly known. For example, the LCF metrics
$k \hash S^3 \times S^1$ (\cite{Leb92}, Theorem 8.2; \cite{ES}), and the ASD metrics on $m \hash \left(-\mathbb{CP}^2\right)$ constructed by LeBrun (\cite{Leb91}), are unobstructed.  Our goal in this paper is to provide a criterion for the vanishing of $H^2_{ASD}$ that only involves conformal invariants of the ASD manifold (and in particular does not depend on verifying any properties of the twistor space).   Our main result is the following:

\begin{theorem}  \label{Main} Suppose $(M^4,g)$ is ASD with Yamabe invariant $Y(M^4,[g]) > 0.$  If
\begin{align} \label{sig}
2 \chi(M^4) + 3 \tau(M^4) \geq -\frac{1}{24 \pi^2}Y(M^4,[g])^2,
\end{align}
then $(M^4,g)$ is unobstructed.
\end{theorem}

The proof of Theorem \ref{Main} relies on two key ideas.  The first is that any element $U \in \ker \mathcal{D}^*$ can be associated to a self-dual harmonic two-form $z = z(U) \in \Lambda^2_{+}\left( \mathcal{E}\right)$ taking its values in the {\em adjoint tractor bundle} (see Section \ref{TractorSec}).   Therefore, $z$ satisfies a twisted version of the usual Weitzenb\"ock formula for self-dual (real-valued) two-forms.  This twisted version provides us with two identities for $U$ (see Theorem \ref{BFThm}).

The second key idea is to make a judicious choice of conformal representative in order to show that the condition (\ref{sig}) implies the vanishing of $U$.  As explained in Section \ref{MainProof}, we choose a conformal metric whose scalar curvature satisfies a differential inequality involving the Schouten tensor (see Theorem \ref{LDThm}), and is adapted to the curvature terms appearing in the Weitzenb\"ock formula(s) for $U$.  Curiously, to prove the existence of this metric we consider a modification of the functional determinant of an conformally covariant elliptic operator first computed by Branson and {\O}rsted \cite{BO}.  An robust theory for the existence of critical points of this functional was developed by Chang and Yang \cite{CYAnnals}, and we are able to show that their ideas also give existence for our modified functional.

We remark that the lower bound in (\ref{sig}) can likely be improved, and it is possible that the techniques of this paper can be used to ultimately remove any additional topological assumption.

%\smallskip

\subsection{Organization}  In Section \ref{TractorSec} we provide the necessary background on the tractor bundle and associated connection and metric.  The main result (for our purposes) is Proposition \ref{BochnerProp}, in which we associate to $U \in \ker \mathcal{D}^*$ a twisted SD harmonic two-form $z = z(U) \in \Lambda^2_{+}\left( \mathcal{E}\right)$.  In Section \ref{WFSec} we compute two Weitzenb\"ock formulas for $U$ that follow from this correspondence.   In Section \ref{MainProof} we give the proof of Theorem \ref{Main}, and relegate the PDE aspects of our work to the Appendix. % \smallskip

\subsection{Acknowledgements}  The authors would like to express their sincere thanks to Claude LeBrun for initially suggesting this problem, and for being an invaluable resource during our work.

The first author is supported by the Royal Society of New Zealand
Marsden grant 19-UOA-008. The second author is supported by NSF grant
DMS-2105460 and a Simons Foundation Fellowship in Mathematics, award
923208.

This project began during a visit of the second author to the
Department of Mathematics at the University of Auckland, in September
2022.  The author would like to thank the department for its support
and hospitality.

%%%%%%%%%%%%%%%%%%%%
\section{Background and the interpretation via tractor calculus} \label{TractorSec}
%%%%%%%%%%%%%%%%%%%%%%

\subsection{Some conventions for Riemannian geometry}
\newcommand{\cc}{\boldsymbol{c}}
\newcommand{\ce}{\mathcal{E}}
\newcommand{\cV}{\mathcal{V}}
\newcommand{\si}{\sigma}
\newcommand{\cT}{\mathcal{T}}
\newcommand{\cK}{\mathcal{K}}
\newcommand{\ID}{\mathrm{ID}}
\newcommand{\om}{\omega}
\newcommand{\Up}{\Upsilon}

\newcounter{mnotecount}%[section]
\newcommand{\mnotex}[1]%{}
{\protect{\stepcounter{mnotecount}}$^{\mbox{\footnotesize $\bullet$\themnotecount}}$
\marginpar{%\color{red}%
\raggedright\tiny\em
$\!\!\!\!\!\!\,\bullet$\themnotecount: #1} }
\newcommand{\edz}[1]{\mnotex{#1}}

\newcommand{\lpl}{
  \mbox{$
  \begin{picture}(12.7,8)(-.5,-1)
  \put(2,0.2){$+$}
  \put(6.2,2.8){\oval(8,8)[l]}
  \end{picture}$}}

For index calculations we will use Penrose's abstract index notation,
unless otherwise indicated. In this we write $\ce_a$ and $\ce^a$ as
alternative notations for, respectively the cotangent and tangent
bundles and the contraction $\om (v)$ of a 1-form $\om$ with a tangent
vector $v^a$ is written with a repeated index $\om_a v^a$ (mimicking
the Einstein summation convention). Tensor bundles are denoted then by
adorning the symbol $\ce$ with appropriate indices and
sometimes also indicating symmetries. For example $\ce_{(ab)}$ is the
notation for $S^2T^*M$, the subbundle of symmetric tensors in
$T^*M\otimes T^*M$.

Then our convention for the
Riemann tensor $R_{ab}{}^c{}_d$ is such that
\begin{align} \begin{split} \label{riemdef}
\left[ \nabla_a,  \nabla_b \right] v^c &= R_{ab}{}^c{}_d v^d, \\
\left[ \nabla_a,  \nabla_b \right] \omega_c &= R_{abc}{}^d \omega_d,
\end{split}
\end{align}
where $\nabla_a$ is the Levi-Civita connection of a metric $g_{ab}$, $v$ any tangent vector field, and $\omega$ is any one-form. Using the metric to raise and lower indices we have, for example $R_{abcd}=g_{ce}R_{ab}{}^e{}_d$. This may be decomposed:
\begin{equation} \label{riemdec}
R_{abcd} = W_{abcd} + 2 \left(g_{c[a}P_{b]d} + g_{d[b}P_{a]c}\right),
\end{equation}
where the completely trace-free part $W_{abcd}$ is the {\em Weyl tensor} and $P_{ab}$ is the {\em Schouten tensor}. It follows that
\begin{equation}
P_{ab} :=  \frac{1}{n-2} \left( R_{ab} - \frac{R}{2 \left( n-1 \right)} g_{ab}\right)
\end{equation}
where $R_{bc}= R_{ab}{}^a{}_c$ is the Ricci tensor and its metric trace $R=g^{ab}R_{ab}$ is scalar curvature. We will use $J$ to denote the metric trace of Schouten, i.e. $J := g^{ab} P_{ab}$. Lastly, we have
\begin{equation} \label{bachcotdef}
\begin{aligned}
C_{abc} & := 2 \nabla_{[b}P_{c]a}, \\
B_{ab} & := - \nabla^c C_{abc} + P^{dc} C_{dacb},
\end{aligned}
\end{equation}
where $C_{abc}$ and $B_{ab}$ are {\em Cotton} and {\em Bach} tensors, respectively. It should be noted that the Bianchi identities imply
\begin{equation} \label{bianchi}
 \left( n-3 \right) C_{abc}= \nabla^d W_{dabc}.
\end{equation}

\subsection{The tractor bundle and connection}

To treat and work with objects that are conformally invariant it is natural
work, at least partly, in the setting of conformal manifolds.
Here by a conformal manifold $(M,\cc)$ we mean a smooth manifold of
dimension $n\geq 3$ equipped with an equivalence class $\cc$ of Riemanian
metrics, where $g_{ab}$, $\widehat{g}_{ab} \in \cc$ means that
$\widehat{g}_{ab}=\Omega^2 g_{ab}$ for some smooth positive function
$\Omega$. On a general conformal manifold $(M,\cc)$, there is no
distinguished connection on $TM$. But there is an invariant and
canonical connection on a closely related bundle, namely the conformal
tractor connection on the standard tractor bundle, see
\cite{BEG,CGtams}.

Here we review the basic conformal tractor calculus on
Riemannian and conformal manifolds. See
\cite{BEG,curry-G,GP} for more details.
Unless stated otherwise, calculations will be done with the use of generic $g \in
\cc$.
%% Hence, we will omit the superscript $g$ in the objects
%% determined by this metric, e.g. $\nabla_a$ will be used instead of
%% $\nabla^g_a$.

 On an $n$-manifold $M$ the top exterior power of the tangent bundle $ \Lambda^{n} TM$ is a line bundle. Thus  its square
$\mathcal{K}:=(\Lambda^{n} TM)^{\otimes 2}$  is
canonically oriented and so one can take oriented roots of it: Given
$w\in \mathbb{R}$ we set
\begin{equation} \label{cdensities}
\ce[w]:=\mathcal{K}^{\frac{w}{2n}} ,
\end{equation}
and refer to this as the bundle of conformal densities. For any vector
bundle $\cV$, we write $\cV[w]$ to mean $\cV[w]:=\cV\otimes\ce[w]$.
For example, $\ce_{(ab)}[w]$ denotes the symmetric second tensor power
of the cotangent bundle tensored with $\ce[w]$, i.e. $S^2T^* M\otimes
\ce[w]$ on $M$. On a fixed Riemannian manifold $\cK$ is canonically
trivialised by the square of the volume form, and so $\cK$ and its
roots are not usually needed explicitly. However if we wish to change
the metric conformally, or work on a conformal structure then these
objects become important.

Since each metric in a conformal class determines a trivialisation of
$\cK$, it follows easily that on a conformal structure there is a
canonical section $\bg_{ab}\in \Gamma(\ce_{(ab)})[2]$. This has the
property that for each positive section $\si\in \Gamma(\ce_+ [1])$
(called a {\em scale}) $g_{ab}:=\si^{-2}\bg_{ab}$ is a metric in
$\cc$. Moreover, the Levi-Civita connection of $g_{ab}$ preserves
$\si$ and therefore $\bg_{ab}$. Thus it makes sense to use the
conformal metric to raise and lower indices, even when we are chosing
a particular metric $g_{ab} \in \cc$ and its Levi-Civita connection
for calculations.  It turns out that this simplifies many
computations, and so in this section we will do that without further
mention. (In the subsequent sections we will work with a fixed metric and use that to trivialise density bundles -- so indices will be raised and lowered using the metric.)

Considering Taylor series for sections of $\ce[1]$ one recovers the jet exact sequence at 2-jets,
\begin{equation}\label{J2}
0\to \ce_{(ab)}[1]\stackrel{\iota}{\to} J^2\ce[1]\to J^1\ce[1]\to 0 .
\end{equation}
Note that $J^2\ce[1]$ and its sequence \eqref{J2} are canonical objects on any smooth manifold. But with a
 conformal structure $\cc$ we have the orthogonal decomposition of $\ce_{ab}[1]$  into trace-free and trace parts
\begin{equation}
\ce_{ab}[1]= \ce_{(ab)_0}[1]\oplus \bg_{ab}\cdot \ce[-1] .
\end{equation}
Thus we can canonically quotient $J^2\ce[1]$ by the image of
$\ce_{(ab)_0}[1]$ under $\iota$ (in (\ref{J2})). The resulting
quotient bundle is denoted $\cT^*$ (or $\ce_A$ in abstract indices),
and called the conformal cotractor bundle. Observing that the
jet exact sequence at 1-jets (of $\ce[1]$),
$$ %%\begin{equation}\label{J1}
0\to \ce_{b}[1]\stackrel{\iota}{\to} J^1\ce[1]\to \ce[1]\to 0,
$$ %%\end{equation}
we see at once that $\cT^*$ has a composition series
\begin{equation}\label{filt}
\cT^*=\ce[1]\lpl \ce_a [1] \lpl \ce[-1] ,
\end{equation}
meaning that $\ce[-1]$ is a subbundle of $\cT^*$ and the quotient
of $\cT^*$ by this (which is $J^1\ce[1]$) has $\ce_a [1]$ as a
subbundle, whereas there is a canonical projection $X:\cT^*\to \ce[1]$.
In abstact indices we write $X^A$ for this map and call it the {\em canonical tractor}.

Given a choice of metric $g \in\cc$, the formula
\begin{equation}\label{thomas-D}
  \sigma \mapsto \frac{1}{n} {[D_A \sigma]}_g :=
  \begin{pmatrix}
    \sigma \\
    \nabla_a \sigma \\
    - \frac{1}{n} \left( \Delta + J \right) \sigma
  \end{pmatrix}
\end{equation}
(where $\Delta $ is the Laplacian $\nabla^a\nabla_a$)
gives a second-order differential operator on $\ce[1]$ which is a linear map $J^2 \ce[1] \to \ce[1] \oplus \ce_a [1] \oplus \ce[-1]$ that clearly factors through $\cT^*$ and so determines an isomorphism
\begin{equation}\label{T_isom}
  \cT^* \stackrel{\sim}{\longrightarrow} {[\cT^*]}_g = \ce[1] \oplus \ce_a [1] \oplus \ce[-1].
\end{equation}
%% The tractor defined in (\ref{thomas_D}) will be called the scale tractor corresponding to the scale $\sigma$ and denoted by $I_{\sigma}$, i.e.
%% \begin{equation}
%% I_{\sigma } := \frac{1}{n} D_A \sigma.
%% \end{equation}

In subsequent discussions, we will use~\eqref{T_isom} to split the
tractor bundles without further comment.  Thus, given $g \in \cc$, an
element $V_A$ of $\ce_A$ may be represented by a triple
$(\si,\mu_a,\rho)$, or equivalently by
\begin{equation}\label{Vsplit}
  V_A=\si Y_A+\mu_a Z_A{}^a+\rho X_A.
\end{equation}
The last display defines the algebraic splitting operators
$Y:\ce[1]\to \cT^*$ and $Z :T^*M[1]\to \cT^*$ (determined by the
choice $g_{ab} \in \cc$) which may be viewed as sections $Y_A\in
\Gamma(\ce_A[-1])$ and $Z_A{}^a\in \Gamma(\ce_A{}^a [-1])$.  We call
these sections $X_A, Y_A$ and $Z_A{}^a$ \emph{tractor projectors}.

By construction the tractor bundle is conformally invariant, i.e. determined by $(M,\cc)$ and  indpenedent of any choice of $g\in\cc$. However the splitting \eqref{Vsplit} is not. Considering the transformation of the operator \eqref{thomas-D} determining the splitting we see that if $\widehat{g}=\Omega^2 f$
the components of an invariant section of $\cT^*$ should transform according to:
\begin{equation}\label{ttrans}
[\cT^*]_{\widehat{g}}\ni \left(\begin{array}{c}
\widehat{\sigma}\\
\widehat{\mu}_b\\
\widehat{\rho}
\end{array}\right)=
\left(\begin{array}{ccc}
1 & 0 & 0\\
\Upsilon_b & \delta^c_b & 0\\
-\frac{1}{2}\Upsilon^2 & -\Upsilon^c & 1
\end{array}\right)\left(\begin{array}{c}
\sigma\\
\mu_c\\
\rho
\end{array}\right)~\sim~ \left(\begin{array}{c}
\sigma\\
\mu_b\\
\rho
\end{array}\right)\in [\cT^*]_g,
\end{equation}
where $\Up_a=\Omega^{-1}\nabla_a \Omega$, and converesely this transformation of triples is the hallmark of an invariant tractor section. Equivalent to the last display is the rule for how the algebraic splitting operators transform
\begin{equation}\label{XYZtrans}
\textstyle
\widehat{X}_A=X_A, \quad  \widehat{Z}_A{}^{b}=Z_A{}^{b}+\Up^bX_A, \quad
\widehat{Y}_A=Y_A-\Up_bZ_A{}^{b}-\frac12\Up_b\Up^bX_A \, .
\end{equation}

Given a metric $g \in \cc$, and the corresponding splittings, as above, the tractor connection is given by the formula
 \begin{equation}\label{tr-conn}
 \nabla_a^\cT  \begin{pmatrix}
    \sigma \\
    \mu_b \\
    \rho
  \end{pmatrix} :=\begin{pmatrix}
   \nabla_a \sigma-\mu_a \\
    \nabla_a\mu_b+P_{ab}\si+\bg_{ab}\rho \\
    \nabla_a\rho- P_{ac}\mu^c
  \end{pmatrix} ,
\end{equation}
where on the right hand side the $\nabla$s are the Levi-Civita connection of $g$. Using the transformation of components, as in \eqref{ttrans}, and also the conformal transformation of the Schouten tensor,
\begin{equation}\label{Ptrans}
P^{\widehat{g}}_{\phantom{\widehat{g}}ab}=P_{ab}-\nabla_a\Upsilon_b
+\Upsilon_a\Upsilon_b-\frac{1}{2}g_{ab}\Upsilon_c\Upsilon^c,  \quad \widehat{g}=\Omega^2 g ,
\end{equation}
reveals that the triple on the right hand side transforms as a 1-form
taking values in $\cT^*$ -- i.e. again by \eqref{ttrans} except
twisted by $\ce_a$. Thus the right hand side of \eqref{tr-conn} is the
splitting into slots of a conformally invariant connection
$\nabla^\cT$ on (section of) the bundle $\cT^*$.

There is a nice conceptual origin for the connection \eqref{tr-conn}. Using \eqref{Ptrans} and the transformation of the Levi-Civita connection it is straightforward to verify that the equation
\begin{equation}\label{AE}
\nabla_{(a}\nabla_{b)_0} \si+ P_{(ab)_0}\si =0
\end{equation}
on conformal densities $\si\in \Gamma(\ce[1])$ is conformally
invariant. As this is an overdetermined PDE, solutions in general do
not exist.  Overdetermined linear PDE are typically studied by
prolongation and it is quickly verified that the tractor parallel
transport given by \eqref{tr-conn} is exactly the closed system that
arises from prolonging \eqref{AE} (see \cite{BEG,curry-G}). From this
observation, and the formula \eqref{tr-conn}, it follows that
non-trivial solutions of \eqref{AE} are non-vanishing on an open dense
set (for $M$ connected) on which the metric $\si^{-2}\bg_{ab}$ is
Einstein -- an observation that has a number of applications, see
\cite{curry-G,G-almost} and references therein.

The tractor bundle is also equipped with a conformally invariant signature $(n+1,1)$  metric $h_{AB} \in \Gamma \left(\ce_{(AB)} \right)$, defined as quadratic form by the mapping
\begin{equation}
[V_A]_g = \begin{pmatrix}
    \sigma \\
    \mu_a \\
    \rho
  \end{pmatrix} \mapsto \mu_a \mu^a +2  \si \rho =: h \left(V,V \right), and
\end{equation}
with the polarization identity. This is important not only by dint of
its conformal invariance, but it is easily checked that this {\em
  tractor metric} $h$ is preserved by $\nabla_a^\cT$, i.e. $
\nabla_a^\cT h_{AB} =0$. Thus it makes sense to use $h_{AB}$ (and its
inverse) to raise and lower tractor indices, and we do this henceforth
without further comment. In particular $X^A=h^{AB}X_B$ is the
canonical tractor (and hence our use of the same kernel symbol). For computations the table of
Figure \ref{tmf} is useful.
\begin{figure}[ht]
$$
\begin{array}{l|ccc}
& Y^A & Z^{Ac} & X^{A}
\\
\hline
Y_{A} & 0 & 0 & 1
\\
Z_{Ab} & 0 & \delta_{b}{}^{c} & 0
\\
X_{A} & 1 & 0 & 0
\end{array}
$$
\caption{Tractor inner product}
\label{tmf}
\end{figure}
We see that $h$ may be decomposed into a sum of projections
$$
h_{AB}=Z_A{}^aZ_{B}{}^b\bg_{ab}+X_AY_B+Y_AX_B\, .
$$

Finally for this section we note that, of course, the curvature of the tractor
connection $\kappa_{abCD}$ is determined by
$$%%\begin{equation} \label{trc1}
2 \nabla_{[a}^\cT \nabla_{b]}^\cT V^C = \kappa_{ab}{}^{C}{}_D V^D \quad \mathrm{for \ all} \quad V^C \in \Gamma \left(\ce^A \right),
$$%%\end{equation}
and can be written in terms of tractor projectors as
\begin{equation} \label{trc2}
  \kappa_{abCD} =W_{abcd} Z_{C}{}^c Z_D{}^d
+  C_{cab} Z_{C}{}^c X_D - C_{cab} X_C Z_D{}^c  \, .
\end{equation}
The bundle $\Lambda^2\cT$ is often termed the {\em adjoint tractor
  bundle}, as it is a vector bundle modelled on the Lie algebra of the
conformal group $SO(n+1,1)$. So, as expected, the tractor curvature is a 2-form taking values in this bundle.

\subsection{The differential splitting operator}\label{tsplit}

\newcommand{\cW}{\mathcal{W}}
\newcommand{\cD}{\mathcal{D}}

In dimensions $n\geq 4$ the tractor curvature is the image of a conformally invariant operator $z$ acting on the Weyl curvature. The operator is as follows:
\begin{lemma} \label{split-lemma} In any dimension $n\geq 4$ there is a conformally invariant differential map
  \begin{equation}\label{split-map}
z: \Gamma(\cW) \to \Gamma(\Lambda^2\otimes \Lambda^2\cT) ,
    \end{equation}
  given by
 \begin{align} \label{zdef} \begin{split}
     U_{ij}{}^k{}_\ell\mapsto z_{ijAB} &= U_{ij}^{\ \ k\ell} Z_{kA} Z_{\ell B}
     - \frac{1}{n-3}(\delta U)^k_{\ ij} \left( X_A Z_{kB} - Z_{kA} X_B \right)  \\
&= \frac{1}{2} U_{ij}^{\ \ k\ell} \left( Z_{kA} Z_{\ell B} - Z_{\ell A} Z_{k B} \right) - \frac{1}{n-3}(\delta U)^k_{\ ij} \left( X_A Z_{kB} - Z_{kA} X_B \right),
   \end{split}
 \end{align}
where
 \begin{align*}
 (\delta U)_{kij} = \nabla^m U_{mkij}.
 \end{align*}
    \end{lemma}
The conformal invariance is easily verified directly using
\eqref{XYZtrans} and the conformal transformation of the Levi-Civita
connection. It can also be deduced from the conformal invariance of
the tractor curvature $\kappa_{abCD}$. In fact the map
\eqref{split-map} is a standard ``BGG-splitting operator'', as in the
theory \cite{CSSannals,CD}, and the conformal deformation
sequence can de understood as arising from a twisting by
$\Lambda^2\cT$ of the de Rham complex \cite{Cap-def,GP-def}.

We are, in particular, interested in the case of dimension $n=4$. Then
it is evident, from the formula \eqref{zdef}, that if $U$ is SD (or ASD)
then so is $z(U)$ as a tractor-twisted 2-form, as the Hodge-$\star$
commutes with the Levi-Civita connection. We obtain more if also $U \in \ker \mathcal{D}^{*}$.

\begin{proposition} \label{BochnerProp}  Suppose $(M^4,g)$ is ASD, and $U$ is SD. Then  $U \in \ker \mathcal{D}^{*}$ iff $z=z(U) \in \Lambda^2_{+}(\mathcal{E})$ is a harmonic, self-dual two form.
\end{proposition}

\begin{proof}
  First note that by standard BGG theory
  \cite{CSSannals,CD,ES-semi} this is true in the conformally flat
  setting.   Now consider the curved case.

  Certainly $\delta z$ is conformally invariant as it is a twisting of
  the usual divergence of 2-forms (in dimension 4) with the
  conformally invariant tractor connection. Thus, starting from the
  top (meaning from the left in the filtration on the adjoint tractor
  bundle that is induced from \eqref{filt}), the first non-zero slot
  of $\delta z$ must be conformally invariant and constructed from
  $U$, its covariant derivatives and possibly curvature contracted
  with $U$. From order considerations the last of these can only
  happen in the bottom slot. The very top slot is rank 2 and involves
  no derivatives of $U$. Thus it is zero, as $U$ is trace-free. At the
  next level we have one derivative of $U$, but it is well known that
  in dimension 4 there is no such conformal invariant.

  Thus the image of $\delta z$ lies in the bottom slot which contains
  a rank 2 tensor. Considering the knowledge of conformally flat case,
  it follows that this conformally invariant object must be a multiple
  of $\cD^* U$, plus possibly another conformally invariant rank 2
  tensor constructed by contracting curvature into $U$. It is easily
  checked that it is not possible to construct a rank 2 conformal
  invariant by contracting curvature into $U$, as $U$ is SD while the
  Weyl curvature is, by assumption, ASD.  Thus the only possibility is
  that the bottom slot is a non-zero multiple of $\cD^* U$.
    \end{proof}
We will give another proof of this result (by direct calculation) in
the next section; see Corollary \ref{dzUCor}.

\begin{corollary}  \label{WFCor} (See \cite{BL}, Theorem 3.10; \cite{GKS}, Lemma 2.1) Suppose $(M^4,g)$ is ASD and $U \in \ker \mathcal{D}^{*}$.  Let $z$ be defined as in (\ref{zdef}).  Then
\begin{align} \label{WF}
\Delta z = \frac{1}{3} R z = 2 J z.
\end{align}
\end{corollary}

\begin{proof} Since $z$ satisfies a twisted version of the Weitzenb\"ock formula (\ref{WFHodge}), when $(M^4,g)$ is ASD the only non-zero curvature term is given by $\frac{1}{3}R$ in both versions. \end{proof}

%%%%%%%%%%%%%%%%%
\section{Weitzenb\"ock formula(s)} \label{WFSec}
%%%%%%%%%%%%%%%%

In this section we use Proposition \ref{BochnerProp} and Corollary \ref{WFCor} to prove a Weitzenb\"ock formula for $z=z(U) \in \Lambda^2_{+}(\mathcal{E})$ when $U \in \ker \mathcal{D}^{*}$.   We begin with some more general calculations.

\begin{proposition} \label{DzProp} Let $U \in \Gamma(\mathcal{W}^{+})$ and let $z = z(U)$ be given by (\ref{zdef}).   Then the covariant derivative of $z$ is given by
\begin{align} \label{Dz} \begin{split}
 \nabla_m &z_{ij AB} = \left(\delta U\right)_{mij} \left( X_A Y_B - Y_A X_B \right) + \left( - \nabla_m \left( \delta U\right)^{\ell}_{\ ij} - U_{ij}^{\ \ k  \ell} P_{m k}\right) \left( X_A Z_{\ell B} - Z_{\ell A} X_B \right) \\
& \ \ \ \ + U_{ij \ m}^{\ \ k } \left( Y_A Z_{k B} - Z_{k A} Y_B \right) + \left( \frac{1}{2} \nabla_m U_{ij}^{\ \ k\ell} + \left( \delta U \right)^k_{\ ij} \, \delta^{\ell}_{m} \right) \left( Z_{kA} Z_{\ell B} - Z_{k B} Z_{\ell A} \right).
 \end{split}
 \end{align}

 \end{proposition}

 For the proof of this and the following proposition we will use the following formulas (see (6) from \cite{GP}):
 \begin{align} \label{DTable} \begin{split}
 \nabla_k X_A &= Z_{kA}, \\
 \nabla_k Z_{A\ell} &= - P_{k\ell} X_A - Y_A g_{k\ell}, \\
 \nabla_k Y_A &= P_{k\ell} Z_A^{\ell}.
 \end{split}
 \end{align}

 \begin{lemma} \label{DBLemma} The following formulas hold:
 \begin{align} \label{Dxy}
 \nabla_m \left( X_A Y_B - Y_A X_B \right) =  - \left( Y_A Z_{mB} - Z_{mA} Y_B \right) + \left(  P_m^q X_A  Z_{qA}   - P_m^q Z_{qA}  X_B \right),
 \end{align}
 \begin{align} \label{Dxz}
 \nabla_m \left( X_A Z_{kB} - Z_{kA} Z_B \right) = - g_{km} \left( X_A Y_B - Y_A X_B \right) - \left( Z_{kA} Z_{mB} - Z_{mA} Z_{kB}  \right), \end{align}
 \begin{align} \label{Dyz}
 \nabla_m \left( Y_A Z_{kB} - Z_{kA} Y_B \right) =  P_{mk} \left( X_A Y_B - Y_A X_B \right) + P_m^q \left( Z_{qA}  Z_{kB} -  Z_{qB} Z_{kA} \right),
 \end{align}
 \begin{align} \label{Dzz} \begin{split}
 \nabla_m \left( Z_{kA} Z_{\ell B} - Z_{k B} Z_{\ell A} \right) &= - P_{mk} \left( X_A Z_{\ell B} - Z_{\ell A} X_B \right) + P_{m\ell} \left( X_A Z_{kB} - Z_{kA} X_B \right) \\
& \ \ - g_{mk} \left( Y_A Z_{\ell B} - Y_B Z_{\ell A} \right) + g_{m \ell} \left( Y_A Z_{kB} - Z_{kA} Y_B \right).
 \end{split}
 \end{align}
 \end{lemma}

 \begin{proof}    First, by the Leibniz rule,
 \begin{align} \label{DB1}
  \nabla_m \left( X_A Y_B - Y_A X_B \right) =  \left( \nabla_m X_A \right) Y_B + X_A \left( \nabla_m Y_B \right) - \left( \nabla_m Y_A \right) X_B - Y_A \left( \nabla_m X_B \right).
  \end{align}
  By (\ref{DTable}) we find
 \begin{align} \label{DB2} \begin{split}
  \nabla_m \left( X_A Y_B - Y_A X_B \right) &=  Z_{mA} Y_B + P_{mq} X_A  Z_{qA}   - P_{mq} Z_{qA}  X_B - Y_A Z_{mB} \\
  &= - \left( Y_A Z_{mB} - Z_{mA} Y_B \right) + \left(  P_{mq} X_A  Z_{qA}   - P_{mq} Z_{qA}  X_B \right).
  \end{split}
  \end{align}
Similarly,
\begin{align} \label{DB3} \begin{split}
\nabla_m & \left( X_A Z_{kB} - Z_{kA} X_B \right) = \left( \nabla_m X_A \right) Z_{kB} + X_A  \left( \nabla_m Z_{kB} \right) - \left( \nabla_m Z_{kA} \right) X_B - Z_{kA} \left( \nabla_m X_B \right) \\
&= Z_{mA} Z_{kB} + X_A  \left( - P_{mk} X_B - Y_B g_{mk} \right) - \left( - P_{mk} X_A - Y_A g_{mk} \right) X_B - Z_{kA} Z_{mB}  \\
&= - g_{km} \left( X_A Y_B - Y_A X_B \right) - \left( Z_{kA} Z_{mB} - Z_{mA} Z_{kB}  \right),
\end{split}
\end{align}
\begin{align} \label{DB4} \begin{split}
\nabla_m & \left( Y_A Z_{kB} - Z_{kA} Y_B \right) = \left( \nabla_m Y_A \right) Z_{kB} + Y_A \left( \nabla_m Z_{kB} \right) - \left( \nabla_m Z_{kA} \right) Y_B - Z_{kA} \left( \nabla_m Y_B \right) \\
&= P_{mq} Z_{qA}  Z_{kB} + Y_A \left( - P_{mk} X_B - Y_B g_{mk}  \right) - \left( - P_{mk} X_A - Y_A g_{mk} \right) Y_B -  P_{mq} Z_{qB} Z_{kA} \\
&= P_{mk} \left( X_A Y_B - Y_A X_B \right) + P_{mq} \left( Z_{qA}  Z_{kB} -  Z_{qB} Z_{kA} \right),
\end{split}
\end{align}
\begin{align} \label{DB5} \begin{split}
\nabla_m \left( Z_{kA} Z_{\ell B}  - Z_{k B} Z_{\ell A} \right) &= \left( \nabla_m Z_{kA} \right) Z_{\ell B} + Z_{kA} \left( \nabla_m Z_{\ell B} \right) - \left( \nabla_m Z_{k B} \right) Z_{\ell A} - Z_{kB} \left( \nabla_m Z_{\ell A} \right) \\
&= \left( - P_{mk} X_A - Y_A g_{mk} \right)  Z_{\ell B} + Z_{kA} \left( - P_{m\ell} X_B - Y_B g_{m \ell} \right) \\
& \ \ \ \   - \left( - P_{mk} X_B - Y_B \delta_{mk} \right) Z_{\ell A} - Z_{kB} \left( P_{m\ell} X_A - Y_A \delta_{m \ell} \right) \\
&= - P_{mk} \left( X_A Z_{\ell B} - Z_{\ell A} X_B \right) + P_{m\ell} \left( X_A Z_{kB} - Z_{kA} X_B \right) \\
& \ \ - g_{mk} \left( Y_A Z_{\ell B} - Y_B Z_{\ell A} \right) + g_{m \ell} \left( Y_A Z_{kB} - Z_{kA} Y_B \right).
 \end{split}
 \end{align}

 \end{proof}

\begin{proof}[Proof of Proposition \ref{DzProp}]  By (\ref{zdef}),
\begin{align} \label{Dzp1} \begin{split}
\nabla_m z_{ijAB} &= \nabla_m \Big\{   \frac{1}{2}  U_{ij}^{\ \ k\ell} \left( Z_{kA} Z_{\ell B} - Z_{\ell A} Z_{k B} \right) - (\delta U)^k_{\ ij} \left( X_A Z_{kB} - Z_{kA} X_B \right) \Big\}  \\
&= \frac{1}{2} \nabla_m  U_{ij}^{\ \ k\ell} \left( Z_{kA} Z_{\ell B} - Z_{\ell A} Z_{k B} \right) + \frac{1}{2} U_{ij}^{\ \ k\ell} \nabla_m \left( Z_{kA} Z_{\ell B} - Z_{\ell A} Z_{k B} \right) \\
& \ \ - \nabla_m \left( \delta U\right)^k_{\ ij}  \left( X_A Z_{kB} - Z_{kA} X_B \right) - \left(\delta U\right)^k_{\ ij} \nabla_m   \left( X_A Z_{kB} - Z_{kA} X_B \right) \\
&= \frac{1}{2} \nabla_m  U_{ij}^{\ \ k\ell} \left( Z_{kA} Z_{\ell B} - Z_{\ell A} Z_{k B} \right) + \frac{1}{2} U_{ij}^{\ \ k\ell} \Big\{ - P_{mk} \left( X_A Z_{\ell B} - Z_{\ell A} X_B \right) \\
& \ \  + P_{m\ell} \left( X_A Z_{kB} - Z_{kA} X_B \right) - g_{mk} \left( Y_A Z_{\ell B} - Y_B Z_{\ell A} \right) + g_{m \ell} \left( Y_A Z_{kB} - Z_{kA} Y_B \right) \Big\} \\
& \ \ \ \ - \nabla_m \left( \delta U\right)^k_{\ ij}  \left( X_A Z_{kB} - Z_{kA} X_B \right) - \left(\delta U\right)^k_{\ ij} \Big\{  - g_{km} \left( X_A Y_B - Y_A X_B \right) \\
& \ \ \ \ \ \ \  - \left( Z_{kA} Z_{mB} - Z_{mA} Z_{kB}  \right) \Big\} \\
&= \left(\delta U\right)_{mij} \left( X_A Y_B - Y_A X_B \right) + \left( - \nabla_m \left( \delta U\right)^{\ell}_{\ ij} - U_{ij}^{\ \ k \ell} P_{m k}\right) \left( X_A Z_{\ell B} - Z_{\ell A} X_B \right) \\
& \ \ \ \ + U_{ij \ m}^{\ \ k} \left( Y_A Z_{k B} - Z_{k A} Y_B \right) + \left(  \frac{1}{2} \nabla_m U_{ij}^{\ \ k\ell} + \left( \delta U\right)^k_{\ ij} \, \delta^{\ell}_{m} \right) \left( Z_{kA} Z_{\ell B} - Z_{k B} Z_{\ell A} \right).
\end{split}
\end{align}

 \end{proof}

We can now give another proof of Proposition \ref{BochnerProp}:

\begin{corollary} \label{dzUCor}  Suppose $(M^4,g)$ is ASD, and $U$ is SD. Then  $U \in \ker \mathcal{D}^{*}$ iff $z=z(U) \in \Lambda^2_{+}(\mathcal{E})$ is a harmonic, self-dual two form.
\end{corollary}

\begin{proof} Since $z$ is self-dual, it suffices to show that $\delta z = 0$.  By (\ref{Dz}),
\begin{align} \label{delzz} \begin{split}
\left(\delta z\right)_{j AB} &= g^{im} \nabla_m z_{ij AB} \\
&= g^{im} \left(\delta U\right)_{mij} \left( X_A Y_B - Y_A X_B \right) + g^{im}\left( - \nabla_m \left( \delta U\right)^{\ell}_{\ ij} - U_{ij}^{\ \ k  \ell} P_{m k}\right) \left( X_A Z_{\ell B} - Z_{\ell A} X_B \right) \\
& \ \ \ \ + g^{im} U_{ij \ m}^{\ \ k } \left( Y_A Z_{k B} - Z_{k A} Y_B \right) + g^{im}\left( \frac{1}{2} \nabla_m U_{ij}^{\ \ k\ell} + \left( \delta U \right)^k_{\ ij} \, \delta^{\ell}_{m} \right) \left( Z_{kA} Z_{\ell B} - Z_{k B} Z_{\ell A} \right).
 \end{split}
 \end{align}
Since $U$ is a curvature-type tensor for which all contractions vanish, it follows that
 \begin{align} \label{delzz2}
\left(\delta z\right)_{j AB} =  - \frac{1}{2} \left(\mathcal{D}^{*}U\right)_{j \ell} \left( X_A Z_{\ell B} - Z_{\ell A} X_B \right).
 \end{align}

\end{proof}

 \begin{proposition} \label{RzProp} Same assumptions as Proposition \ref{DzProp}.  Then the rough laplacian of $z$, i.e., $\Delta z_{ij AB} = g^{k \ell} \nabla_k \nabla_{\ell} z_{ij AB}$, is given by
 \begin{align} \label{Roughz} \begin{split}
 \Delta &z_{ij AB} = \Big\{  \left(W^{+}\right)_{pqi}^{\ \ \ k} U^{qp}_{\ \  k j} +  \left(W^{+}\right)_{pqj}^{\ \ \ k}U^{pq}_{\ \  ik} \Big\}   \left( X_A Y_B - Y_A X_B \right) \\
  &+ \Big\{- \Delta \left( \delta U\right)^{\ell}_{\ ij} - 2 \nabla^m U_{ij}^{\ \ k \ell} P_{m k} - U_{ij}^{\ \ k \ell} \nabla_k J + J \left( \delta U\right)^{\ell}_{\ ij} \Big\}  \left( X_A Z_{\ell B} - Z_{\ell A} X_B \right) \\
   &+ \Big\{ \frac{1}{2} \Delta U_{ij}^{\ \ k\ell} -  \nabla^{k} \left( \delta U \right)^{\ell}_{\ ij}  +  \nabla^{\ell} \left( \delta U \right)^k_{\ ij} - U_{ij}^{\ \ km} P^{\ell}_{m} + U_{ij}^{\ \  \ell m} P^{k}_m \Big\} \times \left( Z_{kA} Z_{\ell B} - Z_{k B} Z_{\ell A} \right).
 \end{split}
 \end{align}
 \end{proposition}

\begin{proof} By (\ref{Dz}),
 \begin{align} \label{Dlz1} \begin{split}
 \Delta z_{ij AB}  &= \nabla^m \left( \nabla_m z_{ij AB}\right) \\
  &= \nabla^m \Big\{ \left(\delta U\right)_{mij} \left( X_A Y_B - Y_A X_B \right) + \left( - \nabla_m \left( \delta U\right)^{\ell}_{\ ij} - U_{ij}^{\ \ k \ell} P_{m k}\right) \left( X_A Z_{\ell B} - Z_{\ell A} X_B \right) \\
& \ \ \ \ + U_{ij \ m}^{\ \ k} \left( Y_A Z_{k B} - Z_{k A} Y_B \right) + \left( \frac{1}{2} \nabla_m U_{ij}^{\ \ k\ell} + \left( \delta U \right)^k_{\ ij} \, \delta^{\ell}_{m} \right) \left( Z_{kA} Z_{\ell B} - Z_{k B} Z_{\ell A} \right) \Big\} \\
&= \left( \nabla^m \left(\delta U\right)_{mij}\right)  \left( X_A Y_B - Y_A X_B \right) +  \left( - \Delta \left( \delta U\right)^{\ell}_{\ ij} - \nabla^m U_{ij}^{\ \ k \ell} P_{m k} - U_{ij}^{\ \ k \ell} \nabla^m P_{m k}\right)  \\
& \times \, \left( X_A Z_{\ell B} - Z_{\ell A} X_B \right) + \left( \nabla^m U_{ij\ m}^{\ \ k}\right) \left( Y_A Z_{k B} - Z_{k A} Y_B \right) + \left( \frac{1}{2} \Delta U_{ij}^{\ \ k\ell} + \nabla^{\ell} \left( \delta U \right)^k_{\ ij}  \right) \\
& \times \,  \left( Z_{kA} Z_{\ell B} - Z_{k B} Z_{\ell A} \right)  + \left(\delta U\right)_{mij} \nabla^m \left( X_A Y_B - Y_A X_B \right)\\
&  + \left( - \nabla_m \left( \delta U\right)^{\ell}_{\ ij} - U_{ij}^{\ \ k \ell} P_{m k}\right) \nabla^m \left( X_A Z_{\ell B} - Z_{\ell A} X_B \right) \\
& \ \ + U_{ij \ m}^{\ \ k} \nabla^m \left( Y_A Z_{k B} - Z_{k A} Y_B \right) + \left( \frac{1}{2} \nabla_m U_{ij}^{\ \ k\ell} + \left( \delta U \right)^k_{\ ij} \, \delta^{\ell}_{m} \right) \nabla^m \left( Z_{kA} Z_{\ell B} - Z_{k B} Z_{\ell A} \right).
 \end{split}
 \end{align}
Then applying the formulas from Lemma \ref{DBLemma} and using the Bianchi identity to write $\nabla^m P_{mk} = \nabla_k J$, we get
\begin{align} \label{Dlz2} \begin{split}
 \Delta & z_{ij AB} = \left( \nabla^m \left(\delta U\right)_{mij}\right)  \left( X_A Y_B - Y_A X_B \right) +  \left( - \Delta \left( \delta U\right)^{\ell}_{\ ij} - \nabla^m U_{ij}^{\ \ k \ell} P_{m k} - U_{ij}^{\ \ k \ell} \nabla_k J\right)  \\
& \times  \left( X_A Z_{\ell B} - Z_{\ell A} X_B \right) + \left( \nabla^m U_{ij\ m}^{\ \ k}\right) \left( Y_A Z_{k B} - Z_{k A} Y_B \right) + \left( \frac{1}{2} \Delta U_{ij}^{\ \ k\ell} + \nabla^{\ell} \left( \delta U \right)^k_{\ ij}  \right) \\
& \times   \left( Z_{kA} Z_{\ell B} - Z_{k B} Z_{\ell A} \right)  + \left(\delta U\right)_{mij} \Big\{  - \left( Y_A Z^m_{B} - Z^m_{A} Y_B \right)+ \left(  P^m_q X_A  Z^q_{A}   - P^m_q Z^q_A  X_B \right) \Big\} \\
&  + \left( - \nabla_m \left( \delta U\right)^{\ell}_{\ ij} - U_{ij}^{\ \ k \ell} P_{m k}\right) \Big\{  - \delta_{\ell}^m  \left( X_A Y_B - Y_A X_B \right) - \left( Z_{\ell A} Z^m_{B} - Z^m_{A} Z_{\ell B}  \right) \Big\} \\
& \ \ + U_{ij\ m}^{\ \ k} \Big\{  P^m_{k} \left( X_A Y_B - Y_A X_B \right) + P^m_{q} \left( Z_{qA}  Z_{kB} -  Z_{qB} Z_{kA} \right) \Big\}\\
& + \left( \frac{1}{2} \nabla_m U_{ij}^{\ \ k\ell} + \left( \delta U \right)^k_{\ ij} \, \delta^{\ell}_{m} \right) \Big\{  - P^m_k \left( X_A Z_{\ell B} - Z_{\ell A} X_B \right) + P^m_{\ell} \left( X_A Z_{kB} - Z_{kA} X_B \right) \\
& \ \ - \delta^m_{k} \left( Y_A Z_{\ell B} - Y_B Z_{\ell A} \right) + \delta^m_{\ell} \left( Y_A Z_{kB} - Z_{kA} Y_B \right) \Big\} \\
&= 2 \left( \nabla^m \left(\delta U\right)_{mij}\right)  \left( X_A Y_B - Y_A X_B \right) + \Big\{- \Delta \left( \delta U\right)^{\ell}_{\ ij} - 2 \nabla^m U_{ij}^{\ \ k \ell} P_{m k} - U_{ij}^{\ \ k \ell} \nabla_k J\\
& \  + J \left( \delta U\right)^{\ell}_{\ ij} \Big\}  \left( X_A Z_{\ell B} - Z_{\ell A} X_B \right) + \Big\{ \frac{1}{2} \Delta U_{ij}^{\ \ k\ell} + 2 \nabla^{\ell} \left( \delta U \right)^k_{\ ij} + U_{ij}^{\ \ mk} P^{\ell}_{m} + U_{ij}^{\ \  \ell m} P^{k}_m \Big\}  \\
& \times \left( Z_{kA} Z_{\ell B} - Z_{k B} Z_{\ell A} \right).
 \end{split}
 \end{align}
 Since the very last term is skew-symmetric in $k, \ell$, we can rewrite this as
 \begin{align} \label{symkl} \begin{split}
 \Delta &z_{ij AB} = 2 \left( \nabla^m \left(\delta U\right)_{mij}\right)  \left( X_A Y_B - Y_A X_B \right) + \Big\{- \Delta \left( \delta U\right)^{\ell}_{\ ij} - 2 \nabla^m U_{ij}^{\ \ k \ell} P_{m k} - U_{ij}^{\ \ k \ell} \nabla_k J\\
& \  + J \left( \delta U\right)^{\ell}_{\ ij} \Big\}  \left( X_A Z_{\ell B} - Z_{\ell A} X_B \right) + \Big\{ \frac{1}{2} \Delta U_{ij}^{\ \ k\ell} -  \nabla^{k} \left( \delta U \right)^{\ell}_{\ ij}  +  \nabla^{\ell} \left( \delta U \right)^k_{\ ij} \\
& \ \  - U_{ij}^{\ \ km} P^{\ell}_{m} + U_{ij}^{\ \  \ell m} P^{k}_m \Big\} \times \left( Z_{kA} Z_{\ell B} - Z_{k B} Z_{\ell A} \right).
 \end{split}
 \end{align}
 The proposition will follow, once we prove

\begin{lemma} \label{deldelLemma}
\begin{align}  \label{dvf}
 \nabla^m \left(\delta U\right)_{mij}= \frac{1}{2}   \left(W^{+}\right)_{pqi}^{\ \ \ k} U^{qp}_{\ \  k j} - \frac{1}{2}  \left(W^{+}\right)_{pqj}^{\ \ \ k}U^{qp}_{\ \  ki}.
 \end{align}
 In particular, if $(M^4,g)$ is ASD, then
 \begin{align} \label{sksm}
 \nabla^m \left(\delta U\right)_{mij}= 0.
 \end{align}
 \end{lemma}

 \begin{proof}  By definition of the divergence and the fact that $U_{\ell m i j } = - U_{m \ell i j}$,
 \begin{align*}
 \nabla^m \left(\delta U\right)_{mij} &= \nabla^m \nabla^{\ell} U_{\ell m i j} \\
 &= g^{m p} g^{ \ell q} \nabla_p \nabla_q  U_{\ell m i j}  \\
 &= \frac{1}{2} g^{m p} g^{ \ell q} \left( \nabla_p \nabla_q - \nabla_q \nabla_p \right) U_{\ell m i j}  \\
 &= \frac{1}{2} g^{m p} g^{ \ell q}\left( R_{pq\ell}^{\ \ \ k}U_{kmij} + R_{pq m}^{\ \ \ k}U_{\ell k ij} + R_{pqi}^{\ \ \ k}U_{\ell m k j}
 + R_{pqj}^{\ \ \ k}U_{\ell m i k }  \right)  \\
 &= \frac{1}{2} \left( - R^{mk} U_{kmij} + R^{\ell k} U_{\ell k ij} \right) + \frac{1}{2}g^{m p} g^{ \ell q} \left( R_{pqi}^{\ \ \ k}U_{\ell m k j}
 + R_{pqj}^{\ \ \ k}U_{\ell m i k }  \right) \\
 &= \frac{1}{2}g^{m p} g^{ \ell q} \left( R_{pqi}^{\ \ \ k}U_{\ell m k j}
 + R_{pqj}^{\ \ \ k}U_{\ell m i k }  \right).
 \end{align*}
Using the decomposition of the curvature tensor as in (\ref{riemdec}), it follows that
\begin{align*}
 \nabla^m \left(\delta U\right)_{mij} &= \frac{1}{2}g^{m p} g^{\ell q} \left( W_{pqi}^{\ \ \ k} + g_{ip} P_q^k - \delta_p^k P_{iq} - g_{iq} P_{p}^k + \delta_{q}^k P_{ip} \right) U_{\ell m k j} \\
 & \ \ \ \  + \frac{1}{2}g^{m p} g^{\ell q} \left( W_{pqj}^{\ \ \ k} + g_{jp} P_q^k - \delta_p^k P_{jq} - g_{jq} P_{p}^k + \delta_{q}^k P_{jp} \right) U_{\ell m ik} \\
 &= \frac{1}{2}   W_{pqi}^{\ \ \ k} U^{qp}_{\ \  k j} + \frac{1}{2} \left( \delta_i^m P^{k \ell} - g^{km} P_{i}^{\ell} - \delta_i^{\ell} P^{km} + g^{k \ell} P_i^m \right) U_{\ell m k j} \\
 &  \ \ \ \ + \frac{1}{2}  W_{pqj}^{\ \ \ k}U^{qp}_{\ \  ik}   + \frac{1}{2} \left( \delta_j^m P^{k \ell} - g^{km} P_{j}^{\ell} - \delta_j^{\ell} P^{km} + g^{k \ell} P_j^m \right)U_{\ell m ik}.
 \end{align*}
 Using the fact that all contractions of $U$ vanish, the above simplifies to
 \begin{align*}
 \nabla^m \left(\delta U\right)_{mij} &= \frac{1}{2}   W_{pqi}^{\ \ \ k} U^{qp}_{\ \  k j} + \frac{1}{2}  W_{pqj}^{\ \ \ k}U^{qp}_{\ \  ik}   + \frac{1}{2} \left( P^{k \ell} U_{\ell i k j} - P^{km} U_{imkj} + P^{k\ell} U_{\ell j ik} - P^{km} U_{jmik} \right) \\
 &= \frac{1}{2}   W_{pqi}^{\ \ \ k} U^{qp}_{\ \  k j} + \frac{1}{2}  W_{pqj}^{\ \ \ k}U^{qp}_{\ \  ik} \\
 &= \frac{1}{2}   W_{pqi}^{\ \ \ k} U^{qp}_{\ \  k j} - \frac{1}{2}  W_{pqj}^{\ \ \ k}U^{qp}_{\ \  ki}
 \end{align*}
 (note that all terms involving the Schouten tensor can be seen to cancel after re-indexing).  Finally, since $U \in \Gamma(\mathcal{W}^{+})$, (\ref{dvf}) follows.
  \end{proof}

\end{proof}

\begin{theorem} \label{BFThm}  If $(M^4,g)$ is ASD and $U \in \ker \mathcal{D}^{*}$, then
\begin{align} \label{I}
\Delta \left( \delta U\right)_{\ell ij} = - 2 \nabla^m U_{ijk \ell} P^k_m - U_{ijk\ell} \nabla^k J + 3 J \left( \delta U\right)_{\ell ij},
\end{align}
\begin{align} \label{II}
\frac{1}{2} \Delta U_{ijk\ell} = \nabla_{k} \left( \delta U \right)_{\ell ij} - \nabla_{\ell} \left( \delta U \right)_{kij} +  U_{ijkm} P_{\ell}^{ m} - U_{ij \ell m} P^m_k + J U_{ijk \ell}.
\end{align}
\end{theorem}

\begin{proof}  If $U \in  \ker \mathcal{D}^{*}$ then $z \in \Lambda^2_{+}(\mathcal{E})$ is harmonic. Therefore, (\ref{I}) and (\ref{II}) follow from Proposition \ref{RzProp} and Corollary \ref{WFCor}.
\end{proof}

\begin{remark}  Although we will not provide a proof, it is not difficult to show that (\ref{II}) holds for any section $U \in \Gamma\left(\mathcal{W}^{+}\right)$; i.e., the condition $U \in \ker \mathcal{D}^{*}$ is not necessary.  However, this is not the case for (\ref{I}).
\end{remark}

\medskip

%%%%%%%%%%%%%%%%%
\section{Proof of Theorem \ref{Main}} \label{MainProof}
%%%%%%%%%%%%%%%%%%%%%%

%In this section we give the proof of Theorem \ref{Main}.

Let $Q_g$ denote the $Q$-curvature of $(M^4,g)$:
\begin{align} \label{Qdef} \begin{split}
Q_g &= \frac{1}{12} \left( - \Delta_g R + R_g^2 - 3 |Ric_g|^2 \right) \\
&= -\frac{1}{2} \Delta_g J + J^2 - |P|^2.
\end{split}
\end{align}
The total $Q$-curvature is a conformal invariant, and we can rewrite (\ref{sig}) as a condition on the total $Q$-curvature and the Yamabe invariant as follows.  By the Chern-Gauss-Bonnet formula,
\begin{align} \label{CGB}
\int Q_g \, dv_g = 4 \pi^2 \chi(M^4) - \frac{1}{8} \int |W_g|^2 \, dv_{g}.
\end{align}
Since $g$ is ASD, (\ref{CGB}) becomes
\begin{align} \label{CGB2}
\int Q_g \, dv_g = 4 \pi^2 \chi(M^4) - \frac{1}{8} \int |W^{-}_{g}|^2 \, dv_{g}.
\end{align}
By the Hirzebruch signature formula,
\begin{align} \label{SF}
48 \pi^2 \tau(M^4) = \int \left( |W^{+}_{g}|^2 - |W^{-}_{g}|^2 \right) \, dv_{g} = - \int |W^{-}_{g}|^2 \, dv_{g}.
\end{align}
Combining this with (\ref{CGB2}) we see that
\begin{align} \label{Qt}
\int Q_g \, dv_g = 2 \pi^2 \left( 2 \chi(M^4) + 3 \tau(M^4) \right).
\end{align}
Therefore, the assumption (\ref{sig}) is equivalent to
\begin{align} \label{Qnn}
\int Q_g \, dv_g \geq - \frac{1}{12}Y(M^4,[g])^2.
\end{align}
%Also, we will restrict to the case where
%\begin{align} \label{Qass}
%\int Q_g \, dv_g \leq 0.
%\end{align}
%We remark that our methods can be used to prove the vanishing of $U$ when $2\chi(M^4) + 3 \tau(M^4) > 0$, but we will not give the details.
\smallskip

\begin{proof}[The proof of Theorem \ref{Main}]  Since our assumptions are conformally invariant, it suffices to prove the result when $g$ is replaced by any metric in the same conformal class.  The metric we will use is a solution of a variational problem related to the regularized determinant of an elliptic operator.  The precise formulation is contained in Theorem \ref{Pexist} in the Appendix, while the following consequence will suffice for our purposes:

\begin{theorem} \label{LDThm} Let $(M^4,g_0)$ be a closed Riemannian four-manifold with  \\

\noindent $(i)$  $Y(M^4, [g_0]) > 0$, \\

\noindent $(ii)$ $\int_{M^4} Q_{g_0} \, dv_{g_0} \geq -\frac{1}{12} Y(M^4, [g_0])^2$.   \\

Then there is a smooth, unit volume conformal metric $g \in [g_0]$ with $J_g = \tr P_g > 0$ satisfying
\begin{align} \label{keyPDE}
\Delta J_g \leq - |\tfP_g|^2 + \frac{15}{4}J_g^2,
\end{align}
where $\tfP_g = P - \frac{1}{4}J g$ is the trace-free Schouten tensor.
\end{theorem}

The proof of this result is somewhat involved, and will also be given in the Appendix.  In the following we will show how Theorem \ref{Main} follows from Theroem \ref{LDThm} and the Weitzenb\"ock formulas of the preceding section.  For the rest of the proof the metric $g$ is assumed to satisfy the conclusions of Theorem \ref{LDThm}, and we will usually suppress the subscript $g$.

 Assume $U \in \ker \mathcal{D}^{*}$.  We will record two integral identities that follow from Theorem \ref{BFThm} and the inequality (\ref{keyPDE}).  First, by (\ref{II}),
\begin{align} \label{BS1} \begin{split}
\frac{1}{2} \Delta |U|^2 &= \langle U , \Delta U \rangle + |\nabla U|^2 \\
&= U^{ijk\ell} \Delta U_{ijk\ell} + |\nabla U|^2 \\
&= U^{ijk\ell} \Big\{ 2 \nabla_{k} \left( \delta U \right)_{\ell ij} - 2 \nabla_{\ell} \left( \delta U \right)_{kij} + 2  U_{ijkm} P_{\ell}^{m} - 2 U_{ij \ell m} P^m_{k} + 2 J U_{ijk \ell} \Big\} +  |\nabla U|^2 \\
&= 4 U^{ijk\ell}  \nabla_{k} \left( \delta U \right)_{\ell ij}  + 4 U^{ijk\ell} \,  U_{ijkm} P_{\ell}^{ m}   + 2 J |U|^2 + |\nabla U|^2.
\end{split}
\end{align}
Recalling that
\begin{align*}
U^{ijk\ell} U_{ijkm} = \frac{1}{4} |U|^2 \delta^{\ell}_{m},
\end{align*}
(\ref{BS1}) implies
\begin{align} \label{BS2}
\frac{1}{2} \Delta |U|^2 = 4 U^{ijk\ell}  \nabla_{k} \left( \delta U \right)_{\ell ij}  + 3 J |U|^2 + |\nabla U|^2.
\end{align}
If we multiply both sides by $J$ and integrate over $M^4$, then
\begin{align} \label{BS3}
\frac{1}{2} \int_{M^4} J \Delta |U|^2 \, dv = \int_{M^4} \big\{ 4 J U^{ijk\ell}  \nabla_{k} \left( \delta U \right)_{\ell ij}  + 3 J^2 |U|^2 + J |\nabla U|^2 \big\} \, dv.
\end{align}
Integrating by parts on the left and using (\ref{keyPDE}), we find
\begin{align} \label{BS4} \begin{split}
\frac{1}{2} \int_{M^4} J \Delta |U|^2 \, dv &= \frac{1}{2} \int_{M^4} \left( \Delta J \right)|U|^2  \, dv \leq \int_{M^4} \left( - \frac{1}{2}  |\tfP|^2 + \frac{15}{8} J^2 \right) |U|^2 \, dv.
\end{split}
\end{align}
Combining (\ref{BS2}) and (\ref{BS3}),
\begin{align} \label{BS5}
0 \geq \int_{M^4} \big\{ 4 J U^{ijk\ell}  \nabla_{k} \left( \delta U \right)_{\ell ij} + \frac{1}{2}|\tfP|^2 |U|^2 + \frac{9}{8} J^2 |U|^2 + J |\nabla U|^2 \big\} \, dv.
\end{align}

We now use (\ref{I}):
\begin{align} \label{BSd1} \begin{split}
\frac{1}{2} \Delta |\delta U|^2 &= |\nabla \delta U|^2 + \langle \delta U ,\Delta \left( \delta U \right) \rangle \\
&= |\nabla \delta U|^2 + \left( \delta U \right)^{\ell ij} \Delta \left( \delta U \right)_{\ell ij} \\
&= |\nabla \delta U|^2 + \left( \delta U \right)^{\ell ij} \Big\{  - 2 \nabla^m U_{ijk \ell} P^k_{m} - U_{ijk \ell} \nabla^k J + 3 J \left( \delta U\right)_{\ell ij} \Big\} \\
&= |\nabla \delta U|^2 - 2  \left( \delta U \right)^{\ell ij} \nabla^m U_{ijk \ell} P_{m}^k -  \left( \delta U \right)^{\ell ij} U_{ijk \ell} \nabla^k J + 3 J |\delta U|^2.
\end{split}
\end{align}
Integrating this over $M^4$ gives
\begin{align} \label{BSd2}
0 = \int_{M^4} \Big\{  |\nabla \delta U|^2 - 2  \left( \delta U \right)^{\ell ij} \nabla^m U_{ijk \ell} P_m^k - \left( \delta U \right)^{\ell ij}  U_{ijk \ell}\nabla^k J + 3 J |\delta U|^2 \Big\} \, dv.
\end{align}
If we integrate by parts in the second term and use the contracted second Bianchi identity $\nabla^m P_m^k = \nabla^k J$, then
\begin{align} \label{BSd3} \begin{split}
\int_{M^4} - 2  \left( \delta U \right)^{\ell ij} \nabla^m U_{ijk \ell} P_m^k \, dv_g &= \int_{M^4} \Big\{  2  \nabla^m \left( \delta U \right)^{\ell ij} U_{ijk \ell} P_{m}^k + 2  \left( \delta U \right)^{\ell ij} U_{ijk \ell} \nabla^m P_m^k \Big\} \, dv_g \\
&= \int_{M^4} \Big\{  2  \nabla^m \left( \delta U \right)^{\ell ij} U_{ijk \ell} P_{m}^k + 2  \left( \delta U \right)^{\ell ij} U_{ijk \ell} \nabla^k J \Big\} \, dv_g.
\end{split}
\end{align}
Substituting this result back into (\ref{BSd2}) gives
\begin{align} \label{BSd4}
0 = \int_{M^4} \Big\{  |\nabla \delta U|^2   +  2  \nabla^m \left( \delta U \right)^{\ell ij} U_{ijk \ell} P_m^k + U_{ijk \ell} \left( \delta U \right)^{\ell ij} \nabla^k J + 3 J |\delta U|^2 \Big\} \, dv.
\end{align}
Next, integrate by parts in the second-to-last term in (\ref{BSd4}):
\begin{align} \label{BSd5}  \begin{split}
\int_{M^4}  U_{ijk \ell} \left( \delta U \right)^{\ell ij} \nabla^k J \, dv &= \int_{M^4} \Big\{ - U_{ijk \ell} \nabla^k \left( \delta U \right)^{\ell ij} J   - \nabla^k U_{ijk \ell}  \left( \delta U \right)^{\ell ij} J\Big\} \, dv \\
&=\int_{M^4} \Big\{ - U_{ijk \ell} \nabla^k \left( \delta U \right)^{\ell ij} J  - J |\delta U|^2 \Big\} \, dv.
\end{split}
\end{align}
Substituting this back into (\ref{BSd4}) gives
\begin{align} \label{BSd6}
0 = \int_{M^4} \Big\{  |\nabla \delta U|^2  +  2  \nabla^m \left( \delta U \right)^{\ell ij} U_{ijk \ell} P_{m}^{k} - J  U_{ijk \ell} \nabla^k \left( \delta U \right)^{\ell ij}   +  2 J |\delta U|^2  \Big\} \, dv.
\end{align}
Finally, we rewrite the curvature terms above in terms of $\mathring{P}$ and $J$:
\begin{align} \label{BSd7}
0 = \int_{M^4} \Big\{  |\nabla \delta U|^2  +  2  \nabla^m \left( \delta U \right)^{\ell ij} U_{ijk \ell} \mathring{P}_{m k} - \frac{1}{2} J  U_{ijk \ell} \nabla^k \left( \delta U \right)^{\ell ij}   +  2 J |\delta U|^2  \Big\} \, dv.
\end{align}
Multiplying by two and rearranging terms, we find
\begin{align} \label{BSd8}
  \int_{M^4} J  U_{ijk \ell} \nabla^k \left( \delta U \right)^{\ell ij} \, dv =  \int_{M^4} \Big\{ 2  |\nabla \delta U|^2  +  4  \nabla^m \left( \delta U \right)^{\ell ij} U_{ijk \ell} \mathring{P}_{m k}   +   4J |\delta U|^2  \Big\} \, dv.
\end{align}

We want to combine (\ref{BSd8}) with (\ref{BS5}).  To do so, we first use (\ref{BSd8}) to write
\begin{align} \label{BSd9} \begin{split}
\int_{M^4} & 4 J  U_{ijk \ell} \nabla^k \left( \delta U \right)^{\ell ij} \, dv   =  3 \int_{M^4}  J  U_{ijk \ell} \nabla^k \left( \delta U \right)^{\ell ij} \, dv +   \int_{M^4}  J  U_{ijk \ell} \nabla^k \left( \delta U \right)^{\ell ij} \, dv \\
&= \int_{M^4} 3 J  U_{ijk \ell} \nabla^k \left( \delta U \right)^{\ell ij} \, dv +   \int_{M^4} \Big\{ 2 |\nabla \delta U|^2  +  4  \nabla^m \left( \delta U \right)^{\ell ij} U_{ijk \ell} \mathring{P}_{m}^k + 4  J |\delta U|^2  \Big\} \, dv \\
&= \int_{M^4} \Big\{ 2 |\nabla \delta U|^2  +  4  \nabla^m \left( \delta U \right)^{\ell ij} U_{ijk \ell} \mathring{P}_{m}^k +  3 J  U_{ijk \ell} \nabla^k \left( \delta U \right)^{\ell ij} + 4 J |\delta U|^2  \Big\} \, dv
\end{split}
\end{align}
Substituting this into (\ref{BS5}), we have
\begin{align} \label{BS6} \begin{split}
0 &\geq \int_{M^4} \Big\{ 2 |\nabla \delta U|^2  +  4  \nabla^m \left( \delta U \right)^{\ell ij} U_{ijk \ell} \mathring{P}_{m}^{k} +  3 J  U_{ijk \ell} \nabla^k \left( \delta U \right)^{\ell ij} + 4 J |\delta U|^2 \\
& \ \ \ \ \ \ \ \ \  + \frac{1}{2} |\tfP|^2 |U|^2 + \frac{9}{8} J^2 |U|^2 + J |\nabla U|^2 \Big\} \, dv \\
&=  \int_{M^4} \Big\{ 2 |\nabla \delta U|^2  +  4  T_m^k  \nabla^m \left( \delta U \right)^{\ell ij} U_{ijk \ell}  + 4 J |\delta U|^2 + \frac{1}{2} |\tfP|^2 |U|^2 + \frac{9}{8} J^2 |U|^2 + J |\nabla U|^2 \Big\} \, dv,
\end{split}
\end{align}
where
\begin{align} \label{Tdef}
T_{km} =  \mathring{P}_{mk} + \frac{3}{4} J g_{km}.
\end{align}
If we define the tensor
\begin{align} \label{Vdef}
V_{m\ell ij} = T_m^k  U_{ijk\ell},
\end{align}
then the term involving $T$ in (\ref{BS6}) can be estimated (via Cauchy-Schwarz) by
\begin{align*}
\left| 4  T_m^k  \nabla^m \left( \delta U \right)^{\ell ij} U_{ijk \ell} \right| &= 4 \left| \nabla^m \left( \delta U \right)^{\ell ij}  V_{m \ell ij} \right| \\
&\leq 4\left| \nabla \delta U \right| \left| V \right|.
\end{align*}
By the arithmetic-geometric mean inequality,
\begin{align} \label{BS7}
\left| 4  T_{km}  \nabla_m \left( \delta U \right)_{\ell ij} U_{ijk \ell} \right|  \leq 2 \left| \nabla \delta U \right|^2 + 2 \left| V \right|^2.
\end{align}
By the definition of $V$,
\begin{align*}
\left| V \right|^2 &= V^{m \ell ij} V_{m \ell ij } \\
&=  T_p^m  U^{ij p \ell} T_m^k  U_{ijk\ell} \\
&= T_p^m T_m^k U^{ij\ell p} U_{ij\ell k}  \\
&= T_p^m T_m^k \left( \frac{1}{4} |U|^2 \delta_k^p  \right) \\
&= \frac{1}{4} |T|^2 |U|^2 \\
&= \frac{1}{4} \left( |\mathring{P}|^2 + \frac{9}{4} J^2 \right) |U|^2.
\end{align*}
Consequently, (\ref{BS7}) implies
 \begin{align*}
\left| 4  T_{km}  \nabla_m \left( \delta U \right)_{\ell ij} U_{ijk \ell} \right|  \leq 2 \left| \nabla \delta U \right|^2 + \frac{1}{2} |\mathring{P}|^2 |U|^2 + \frac{9}{8} J^2 |U|^2,
\end{align*}
hence
 \begin{align} \label{BS8}
4 T_{km}  \nabla_m \left( \delta U \right)_{\ell ij} U_{ijk \ell}  \geq - 2 \left| \nabla \delta U \right|^2 - \frac{1}{2} |\mathring{P}|^2 |U|^2 - \frac{9}{8} J^2 |U|^2.
\end{align}
Substituting this into (\ref{BS6}), we conclude
\begin{align} \label{BS9}
0 \geq \int_{M^4} \Bigg\{ 4 J |\delta U|^2  + J |\nabla U|^2 \Bigg\} \, dv.
\end{align}
Since $J > 0$, it follows that $\nabla U = 0$.  However, it is then immediate from (\ref{BS5}) that $U \equiv 0$.
\end{proof}

%%%%%%%%%%%%%%
\section{Appendix: The proof of Theorem \ref{LDThm}}
%%%%%%%%%%%%%%%%%

In this Appendix we give the proof of Theorem \ref{LDThm}.  The material in this section is an extension of the existence work of Chang-Yang \cite{CYAnnals} for critical points of the regularized determinant of conformally covariant operators.  We begin with a brief overview of
their work, omitting the motivation from spectral geometry and limiting ourselves to the underlying variational problem.

Let $(M^4,g)$ be a closed, four-dimensional Riemannian manifold, and $W^{2,2}(M)$ the Sobolev space of functions whose weak derivatives up to order two are in $L^2$.  Consider the following functionals on $W^{2,2}(M)$:
\begin{align} \label{Idef}
I^{\pm}[w] = 4 \int w| W^{\pm}_g |^2\, dv_g - \big( \int |W^{\pm}_g|^2\, dv_g \big) \log \fint
e^{4w}\, dv_g,
\end{align}
\begin{align} \label{IIdef}
II[w] = \int w P_gw\, dv_g + 4 \int Q_g w \, dv_g - \big( \int Q_g\, dv_g\big) \log
\fint e^{4w}\, dv_g,
\end{align}
\begin{align} \label{IIIdef}
III[w] = 12 \int (\Delta w + |\nabla w|^2 )^2\, dv_g - 4 \int ( w
\Delta R_g + R_g|\nabla w|^2 )\, dv_g,
\end{align}
where $P_g$ denotes the Paneitz operator, $Q_g$ the scalar curvature, and $\fint$ denotes the normalized integral (i.e., divided by the volume of $g$).   Let $\gamma_1^{\pm},\gamma_2,\gamma_3$ be constants and define $F : W^{2,2}(M) \rightarrow \mathbb{R}$ be given by
\begin{align} \label{Fdef}
F[w] = \gamma_1^{+}I^{+}[w] + \gamma_1^{-} I^{-}[w] + \gamma_2 II[w] + \gamma_3 III[w].
\end{align}
We also define the associated conformal invariant:
\begin{align} \label{kappadef}
\kappa_g = -\gamma_1^{+} \int |W^{+}|^2\, dv_g - \gamma_1^{-} \int |W^{-}|^2 \, dv_g - \gamma_2 \int Q_g\, dv_g.
\end{align}

As explained in \cite{CYAnnals}, critical points of $F$ determine a conformal metric satisfying a fourth order curvature condition.  More precisely, if we define the {\em $U$-curvature} of $g$ by
\begin{align} \label{Udef}
U = U(g) = \gamma_1^{+} |W^{+}_g|^2 + \gamma_1^{-} |W^{-}_g|^2 +\gamma_2 Q_g - \gamma_3 \Delta_g R_g,
\end{align}
then $w$ is a smooth critical point of $F$ if and only if the conformal metric $g_F =
e^{2w}g$ satisfies
\begin{align} \label{EL2}
U(g_F) \equiv \mu
\end{align}
for some constant $\mu$.  A general existence result for critical points of $F$ was proved in \cite{CYAnnals}:

\begin{theorem} \label{CYExist} {\em (See \cite{CYAnnals}, Theorem 1.1; also \cite{GurskyAnnals}, Corollary 1.1)}
Assume:

\vskip.1in \noindent $(i)$ $\gamma_2 < 0$ and $\gamma_3 < 0$,

\vskip.1in \noindent $(ii)$ $\kappa_g < (-\gamma_2)8 \pi^2$.

\vskip.1in  Then $\sup_{w \in W^{2,2}(M)} F[w]$ is attained by some
$w \in W^{2,2}$.
\end{theorem}

Regularity of extremals was proved by the second author in joint work with Chang-Yang
\cite{CGYAJM}; later Uhlenbeck-Viaclovsky proved a more general regularity result for arbitrary
critical points of $F$ (see \cite{UV}).

To prove Theorem \ref{LDThm} we need to introduce another functional:
\begin{align} \label{IVdef}
IV[w] = \dfrac{\int \left( R_g + 6 |\nabla w|^2 \right)e^{2w} \, dv_g}{\left( \int e^{4w} \, dv_g \right)^{1/2}}.
\end{align}
This is just the Yamabe functional, written in a slightly non-standard form:  if $\widetilde{g} = e^{2w}g$, then
\begin{align} \label{4Yamabe}
IV[w] = \dfrac{ \int R_{\widetilde{g}} \, dv_{\widetilde{g}} }{ \mbox{Vol}(\widetilde{g})^{1/2}}.
\end{align}
 Given a constant $\gamma_4$, we define $\Phi : W^{2,2}(M) \rightarrow \mathbb{R}$ by
\begin{align} \label{Phidef}
\Phi[w] = \gamma_1^{+}I^{+}[w] + \gamma_1^{-} I^{-}[w] + \gamma_2 II[w] + \gamma_3 III[w] + \gamma_4 IV[w].
\end{align}
A trivial modification of the existence result of Chang-Yang and the regularity results of Chang-Gursky-Yang and Uhenbeck-Viaclovsky gives the following:

\begin{theorem} \label{Pexist} Assume:

\vskip.1in \noindent $(i)$ $\gamma_2 < 0, \gamma_3 < 0, \gamma_4 \leq 0$.

\vskip.1in \noindent $(ii)$ $\kappa_g < (-\gamma_2)8 \pi^2$.

\vskip.1in  Then $\sup_{w \in W^{2,2}(M)} \Phi[w]$ is attained by some $w \in W^{2,2}(M)$.  Moreover, $w \in C^{\infty}$, and the conformal metric $\widetilde{g} = e^{2w} g$ satisfies
\begin{align} \label{ELPhi}
\gamma_1^{+} |W^{+}_{\widetilde{g}}|^2 + \gamma_1^{-} |W^{-}_{\widetilde{g}}|^2 +\gamma_2 Q_{\widetilde{g}} - \gamma_3 \Delta_{\widetilde{g}} R_{\widetilde{g}} + \frac{1}{2} \gamma_4 R_{\widetilde{g}} = \mu,
\end{align}
for some constant $\mu$.
\end{theorem}

\begin{proof}  Note that the functional $\Phi$ is scale-invariant:
\begin{align*}
\Phi[w + c ] = \Phi[w]
\end{align*}
for any constant $c$.  Therefore, we may normalize a maximizing sequence $\{ w_k \}$ for $\Phi$ so that
\begin{align*}
\int e^{4w_k} \, dv_g = 1.
\end{align*}
Assuming $R_g \geq -C$, it follows from the Schwartz inequality that
\begin{align*}
IV[w] &= \int \left( R_g + 6 |\nabla w_k|^2 \right)e^{2w_k} \, dv_g \\
&\geq  \int R_g  e^{2w_k} \, dv_g \\
&\geq -C \left( \int e^{4w_k} \, dv_g \right)^{1/2} \\
&\geq -C.
\end{align*}
Consequently, if $\gamma_4 \leq 0$ then $\gamma_4 IV[w]$ is bounded above.  By (\ref{4Yamabe}), it is also bounded below (by the Yamabe invariant). Therefore, the addition of this term has no effect on the estimates in the existence proof of Chang-Yang.
\end{proof}

We are now ready to prove Theorem \ref{LDThm}:

\begin{proof}[The proof of Theorem \ref{LDThm}]  We first remark that if $(M^4,g_0)$ is conformally equivalent to the round sphere (suitably normalized), then $g = g_c$ satisfies the conclusions of the Theorem.  Therefore, we may assume $(M^4,g_0)$ is not conformally the round sphere.

Taking
\begin{align*}
\gamma_1^{\pm} &= 0, \\
\gamma_2 &= -6, \\
\gamma_3 &= -\frac{1}{2},  \\
\gamma_4 &= - 2 Y(M^4,[g_0]),
\end{align*}
then
\begin{align*}
\kappa_g =  6 \int Q_{g_0}\, dv_{g_0}.% \leq 0 < (-\gamma_2)8\pi^2 = 48\pi^2.
\end{align*}
To use Theorem \ref{Pexist} we need to verify assumption $(ii)$; i.e.,
\begin{align} \label{Adams}
\int Q_{g_0}\, dv_{g_0} < 8 \pi^2.
\end{align}
By Theorem B of \cite{GurskyCMP}, (\ref{Adams}) holds as long as $(M^4, g_0)$ is not conformally equivalent to the round sphere.  Therefore, by  Theorem \ref{Pexist} there is a smooth conformal metric $g = e^{2w}g_0$ (which we can normalize to have unit volume) satisfying
\begin{align} \label{ELPhi2}
-6 Q_g + \frac{1}{2} \Delta_g R_g   - Y(M^4,[g_0]) R_g = \mu.
\end{align}
For the rest of the proof we will omit the subscript $g$.  If $E = Ric - \frac{1}{4}Rg$ denotes the trace-free Ricci tensor of $g$, then we may use the definition of the $Q$-curvature to rewrite (\ref{ELPhi2}) as
\begin{align} \label{EL3}
\Delta R = \frac{1}{8}R^2 - \frac{3}{2}|E|^2 + Y(M^4,[g_0]) \, R + \mu.
\end{align}
By the arithmetic-geometric mean inequality,
\begin{align*}
Y(M^4,[g_0]) \, R \leq \frac{1}{2} R^2 + \frac{1}{2} Y(M^4,[g_0])^2.
\end{align*}
Therefore,
\begin{align} \label{EL4}
\Delta R \leq \frac{5}{8}R^2 - \frac{3}{2}|E|^2 + \frac{1}{2} Y(M^4,[g_0])^2 + \mu.
\end{align}

\begin{claim} \label{keymu}
\begin{align} \label{keysize}
\mu + \frac{1}{2} Y(M^4,[g_0])^2 \leq 0.
\end{align}
\end{claim}

For now let us assume the claim and see how the theorem follows.

From (\ref{keysize}) and (\ref{EL4}) it follows that
\begin{align} \label{EL5}
\Delta R \leq \frac{5}{8}R^2 - \frac{3}{2}|E|^2.
\end{align}
Using the fact that $J = \frac{1}{6}R$ and $\tfP = \frac{1}{2}E$, this inequality can also be written
\begin{align*}
\Delta J \leq - |\tfP|^2 + \frac{15}{4}J^2,
\end{align*}
hence (\ref{keyPDE}) holds.

To see that $R > 0$, we use (\ref{EL3}) and the fact that $\mu < 0$ to write
\begin{align} \label{ELf} \begin{split}
\Delta R &\leq \frac{1}{8}R^2 + Y(M^4,[g_0]) \, R \\
&\leq \frac{1}{6} R^2 + Y(M^4,[g_0]) \, R.
\end{split}
\end{align}
Let $\phi > 0$ denote the eigenfunction associated to the first eigenvalue $\lambda_1(L)$ of the conformal laplacian:
\begin{align} \label{L1}
L\phi := \left(- 6 \Delta + R\right)\phi = \lambda_1(L).
\end{align}
Since $Y(M^4,[g]) > 0$, it follows that $\lambda_1(L) > 0$.  An easy calculation using (\ref{ELf}) and (\ref{L1}) gives
\begin{align*}
\Delta \dfrac{R}{\phi} \leq - 2 \langle \nabla \left( \frac{R}{\phi} \right), \frac{\nabla \phi}{\phi} \rangle + \left( Y\left(M^4,[g_0]\right) + \frac{1}{6} \lambda_1(L) \right) \frac{R}{\phi}.
\end{align*}
It follows from the strong maximum principle that $R/\phi > 0$ on $M$, hence $R > 0$.   This completes the proof of the theorem, once we prove Claim \ref{keymu}.

\medskip

\begin{proof}[Proof of Claim \ref{keymu}] If integrate (\ref{ELPhi2}) over $M$ and use the fact that $g$ has unit volume, we obtain
\begin{align} \label{tr1}
\mu = -6 \int Q_g \, dv_g  - Y(M^4,[g_0]) \int R_g \, dv_g.
\end{align}
By definition of the Yamabe invariant (again using the fact that $g$ has unit volume) and the fact that $Y(M^4,[g_0]) > 0$,
\begin{align*}
Y(M^4,[g_0]) \int R_g \, dv_g \geq Y(M^4,[g_0])^2.
\end{align*}
Therefore, by (\ref{tr1}),
\begin{align} \label{tr2}
\mu \leq -6 \int Q_g \, dv_g - Y(M^4,[g_0])^2.
\end{align}
Since the total $Q$-curvature is a conformal invariant, using assumption $(ii)$ of the theorem we see that
\begin{align*}
\mu &\leq -6 \int Q_g \, dv_g - Y(M^4,[g_0])^2 \\
&= -6 \int Q_{g_0} \, dv_{g_0} - Y(M^4,[g_0])^2 \\
&\leq \frac{1}{2}Y(M^4,[g_0])^2 - Y(M^4,[g_0])^2 \\
&= -\frac{1}{2} Y(M^4,[g_0])^2,
\end{align*}
which proves (\ref{keysize}).
\end{proof}

\end{proof}

%%%%%%%%%%%%%%%%%%%%%%%%%%%%%
%
%% Cap-def,G-Pet-def
%

\end{document}